\documentclass[a4paper, oneside, 11pt]{amsart} 

\usepackage[english]{babel}
\usepackage[utf8x]{inputenc}
\usepackage[T1]{fontenc}

\usepackage{amsmath}
\usepackage{amssymb}
\usepackage{bbm}
\usepackage{geometry}
\usepackage{verbatim}
\usepackage{graphicx}
\usepackage{hyperref}

\newtheorem{thm}{Theorem}[section]
\newtheorem{cor}[thm]{Corollary}
\newtheorem{lem}[thm]{Lemma}
\newtheorem{prop}[thm]{Proposition}

\theoremstyle{remark}
\newtheorem{rem}[thm]{Remark}

\newcommand\R{\mathbb R} 
\newcommand\Z{\mathbb Z} 

\renewcommand{\d}{\mathrm d}
\newcommand{\pd}[1]{\partial #1}

\def\tend_#1{\mathrel{\mathop{\hbox{\rightarrowfill}}\limits_{\>#1\>}}}

\renewcommand{\leq}{\leqslant}
\renewcommand{\geq}{\geqslant}

\newcommand{\T}{\mathrm T} 
\renewcommand{\S}{\mathrm S} 

\DeclareMathOperator{\Wh}{Wh}

\DeclareMathOperator{\GL}{GL}

\title{Contact manifolds with symplectomorphic symplectizations}
\author{Sylvain Courte}

\begin{document}

\begin{abstract}
We provide examples of contact manifolds of any odd dimension $\geq 5$ which
are not diffeomorphic but have exact symplectomorphic symplectizations.
\end{abstract}

\maketitle

\tableofcontents

\section{Introduction} 

Symplectization provides a bridge between contact and symplectic geometry. It associates to any
contact manifold $(M, \xi)$ (namely any manifold $M$ equipped with a cooriented contact structure $\xi$)
an exact symplectic manifold $(\S_\xi M, \lambda_\xi)$ (that is $\omega_\xi = \d \lambda_\xi$ is a symplectic
form on $\S_\xi M$) diffeomorphic to $\R \times M$. Most of known contact invariants are defined using
symplectizations. For example, the contact homology of $(M, \xi)$ seems to depend only on the
symplectomorphism type of $(\S_\xi M, \omega_\xi)$. Therefore, one might think that if two contact manifolds
have symplectomorphic symplectizations then they are contactomorphic (see \cite{CE2012} p.239 where
the problem is addressed). In this paper, we prove the following theorem which shows that this is not true (see
section 3 for the definition of exact symplectomorphism).

\begin{thm}\label{main}
Let $M$ and $M'$ be closed manifolds of dimension $\geq 5$ such that $\R \times M$ and $\R \times M'$
are diffeomorphic. Then for every contact structure $\xi$ on $M$, there exists a contact structure $\xi'$ on $M'$
such that the symplectizations $\S_\xi M$ and $\S_{\xi'}M'$ are exact symplectomorphic.
\end{thm}

As a concrete example, consider $M = L(7,1) \times S^{2n}$ and $M' = L(7,2) \times S^{2n}$ for $n \geq 1$,
where $L(p,q)$ denotes the three-dimensional lens space of type $(p,q)$. In \cite{MR0133127}, J. Milnor
showed using Reidemeister torsion that $M$ and $M'$ are not diffeomorphic, but proved however that
they are $h$-cobordant. The $s$-cobordism theorem then implies that $\R \times M$ and $\R \times M'$
are diffeomorphic (see section 2). On the other hand $M$ admits a contact structure $\xi$. Indeed, for $n=1$,
$M$ is diffeomorphic to the unit tangent bundle of $L(7,1)$ and in general, $M$ is the boundary of
$L(7,1) \times D^{2n+1}$ which is a Weinstein domain by Y. Eliashberg's work (see section 3).
Theorem \ref{main} above then provides a contact structure $\xi'$ on $M'$ such that $\S_\xi M$ and $\S_{\xi'} M'$ are exact
symplectomorphic, though $M$ and $M'$ are not even diffeomorphic.

The main ingredients in the proof are the flexibility properties of certain Weinstein cobordisms,
first discovered by Y. Eliashberg (\cite{MR1461569}) and developped with K. Cieliebak (\cite{CE2012})
on the base of E. Murphy's work (\cite{M2012a}).

This paper is organized as follows. Section 2 contains some recollections about Morse-Smale theory and the
$s$-cobordism theorem. In section 3, we discuss symplectization of contact manifolds, Weinstein
cobordisms, and quote two theorems from \cite{CE2012} about so-called
\emph{flexible} Weinstein cobordisms. Section 4 contains our results and section 5 discusses
a few open questions.

\bigskip

\noindent{\bf{Acknowledgements.}} I am very grateful to Emmanuel Giroux for his help and
support during this work. I also thank Marco Mazzucchelli for proofreading a first draft
of this paper.
\section{$h$-cobordisms}

Since we look for contact manifolds with symplectomorphic symplectizations, we must
first tackle the following problem from differential topology :
\emph{if $M$ and $M'$ are closed oriented manifolds, when does $\R \times M$ and
$\R \times M'$ are diffeomorphic ?}
If $M$ and $M'$ are $3$-dimensional, there are no known examples where $M$ and $M'$ are
not diffeomorphic (see the last question in section 5). However, in dimension $\geq 5$, there
are examples where $M$ and $M'$ are not diffeomorphic (see the example in the introduction).

Let us introduce some terminology. A \emph{cobordism from} $M$ \emph{to} $M'$ is a triple
$(W; M, M')$ where $W$ is a compact oriented manifold together with a decomposition of its
boundary as $\pd W = \pd_+ W \sqcup \pd_- W$ and orientation-preserving diffeomorphisms
$\pd_- W \to -M$ and $\pd_+ W \to M'$. Here, as customary, $\pd W$ is oriented with outer normal first
convention and $-M$ means $M$ with opposite orientation. We insist that the identification of
the boundary is part of the data (as in \cite{MR0190942}). Given two cobordisms $(W; M, M')$
and $(W'; M', M'')$, we can \emph{compose} them by gluing along $M'$ and get another cobordism
denoted by $(W \odot W'; M, M'')$. Producing an actual smooth structure on $W \odot W'$ requires
some choices but the result is independent of these choices up to a diffeomorphism relative to
the boundary. A cobordism $(W; M, M')$ is called an \emph{$h$-cobordism} if both inclusion maps
$M \to W$ and $M' \to W$ are homotopy equivalences. A \emph{product cobordism}
$([0,1] \times M; M, M)$ is an obvious example. A \emph{Morse function}
on a cobordism $(W; M, M')$ is a smooth function $\phi : W \to \R$ which is constant on the boundary,
satisfies $\d \phi > 0$ on inward pointing vectors at $M$ and outward pointing vectors at $M'$,
and whose critical points are non-degenerate. A \emph{pseudo-gradient vector field} for a Morse
function $\phi$ is a vector field $X$ such that $X . \phi > 0$ outside of the critical points of $\phi$
and such that at each critical point $p$, the linearized vector field $X_p^{lin}$ has no
eigenvalue with vanishing real part. We call $(X, \phi)$ a \emph{Morse pair}. A \emph{Morse
homotopy} is a smooth path $(X_s, \phi_s)$ which is generic in the sense that it encounters only
birth-death type singularities. There are finitely many parameters $s$ where $\phi_s$ has a degenerate
critical points, for any other parameter $s$, $(X_s, \phi_s)$ is a Morse pair. S. Smale showed
in \cite{MR0153022} that simply-connected $h$-cobordisms of dimension $\geq 6$ are
diffeomorphic to product cobordisms. The non-simply connected case is the subject of the
$s$-cobordism theorem, proved by D. Barden, B. Mazur and J. Stallings, which provides a
complete classification of $h$-cobordisms $(W; M, -)$ up to diffeomorphism relative to $M$
in terms of so-called \emph{Whitehead torsion}. These theorems are proved using what is
now called \emph{Morse-Smale theory}. This consist in simplifying Morse pairs by cancelling
critical points. For example, if we are able to cancel all the critical points of a Morse function
on a cobordism, the latter must be diffeomorphic to a product cobordism.

Here are two lemmas from Morse-Smale theory which are building blocks for the proof
of the $s$-cobordism theorem (see \cite{MR0189048}). We will use them in section 4.

\begin{lem}[Normal form]\label{normalform}
Let $(W; M, M')$ be an $h$-cobordism of dimension $\geq 6$. Then there is a Morse pair
with only critical points of index $2$ and $3$.
\end{lem}

We briefly indicate why it is not always possible to cancel the remaining critical points
(see \cite{MR0189048} for more details). Take a Morse pair $(X, \phi)$ given by lemma
\ref{normalform} and lift it to a Morse pair $(\tilde X, \tilde \phi)$ on a universal cover
$\tilde M \to M$. The Morse complex $(C_i, \pd_i)$ associated to $(\tilde X, \tilde \phi)$
is a chain complex over $\Z[\pi_1 M]$ which is only non-zero in degree $2$ and $3$.
Moreover, since $W$ is an $h$-cobordism, this complex is acyclic. Therefore we get a matrix
$A \in \GL(\Z[\pi_1 M])$ which represents the boundary operator $\pd_3 : C_3 \to C_2$.
It turns out that the class of $A$ in a quotient group $\Wh(\pi_1 M)$ of $\GL(\Z[\pi_1 M])$,
called the \emph{Whitehead group} of $\pi_1 M$ is an actual invariant of the $h$-cobordism,
called \emph{Whitehead torsion}. The remaining critical points can be cancelled if and only
if the Whitehead torsion vanishes.

\begin{lem}\label{cancelling}
Let $(W; M, M')$ be an $h$-cobordism of dimension $\geq 6$ with vanishing Whitehead torsion.
Let $(X, \phi)$ be a Morse pair with only critical points of index $2$ and $3$. Then there is a
Morse homotopy $(X_s, \phi_s)$ fixed near the boundary, such that $(X_0, \phi_0) = (X, \phi)$
and $(X_1, \phi_1)$ has no critical points.
\end{lem}

We now state the $s$-cobordism theorem.

\begin{thm}[Barden, Mazur, Stallings; 1965]
Let $M$ be a closed oriented manifold of dimension $\geq 5$.
Whitehead torsion $\tau(W, M) \in \Wh(\pi_1 M)$ of a cobordism $W$ from $M$ induces a
bijective correspondence:
\[\tau : \{\text{$h$-cobordisms $(W; M, -)$ up to diffeomorphism relative to $M$}\} \to \Wh(\pi_1 M)\]
\end{thm}

The reader may consult \cite{MR0189048, MR0196736, MR2061749} for more information
about Whitehead torsion and the $s$-cobordism theorem. We do not go further in this topic
since we will only need the following corollary:

\begin{cor}\label{invertibility}
For any $h$-cobordism $(W; M, M')$ of dimension $\geq 6$, there is an $h$-cobordism
$(W'; M', M)$ such that $W \odot W'$ is diffeomorphic to $[0,1] \times M$ and
$W' \odot W$ is diffeomorphic to $[0,1] \times M'$.
\end{cor}

The reason is that, according to the $s$-cobordism theorem, $h$-cobordisms are classified
by Whitehead torsion which takes value in an abelian group. The "inverse" $h$-cobordism
$W'$ in corollary \ref{invertibility} is essentially the $h$-cobordism with opposite Whitehead
torsion (see \cite{MR0196736}).

In particular, this implies that (oriented) $h$-cobordism between closed oriented manifolds
of dimension $\geq 5$ defines an equivalence relation (the symmetry property was not obvious).

Let $\Psi : \R \times M \to \R \times M'$ be a diffeomorphism. Consider in $\R \times M'$ the
regions between $\{c'\} \times M'$ and $\Psi(\{c\} \times M)$, and between $\{c'\} \times M'$
and $\Psi(\{-c\} \times M)$ for $c$ sufficiently large. These are cobordisms inverse to each
other, so in particular $h$-cobordisms. Conversely, we have the following well-known
corollary of the $s$-cobordism theorem.

\begin{cor}\label{mazurtrick}
Let $M$ and $M'$ be closed oriented manifolds of dimension $\geq 5$. If $M$ and $M'$ are $h$-cobordant, then
$\R \times M$  and $\R \times M'$ are diffeomorphic.
\end{cor}

\begin{proof}
The proof is an instance of the so-called \emph{Mazur trick} which consists in introducing
parentheses in an infinite sum in two different ways.

By corollary \ref{invertibility}, there are $h$-cobordisms $(W; M, M')$ and $(W'; M', M)$
such that:
\[W \odot W' \simeq [0, 1] \times M \text{ and } W' \odot W \simeq [0, 1] \times M'.\]

We now consider the open manifold $V$ obtained by gluing infinitely many copies
of $W$ and $W'$ in an alternate pattern:
\[V = \cdots \odot W' \odot W \odot W' \odot W \odot \cdots \]

The proof can be sumed up formally in one line:

\[\R \times M \simeq \bigodot_{j \in \Z} [j, j+1] \times M \simeq \bigodot_{j \in \Z}(W \odot W')  \simeq V
\simeq \bigodot_{j \in \Z} (W' \odot W) \simeq \bigodot_{j \in \Z} [j, j+1] \times M' \simeq \R \times M'\]
\end{proof}

We finish this section by studying the extension problem of non-degenerate $2$-forms on
$h$-cobordisms. 

\begin{rem}\label{hlemprod}
In the case of a product cobordism $W = [0, 1] \times M$, we can retract $W$
by an isotopy to $[0, \epsilon] \times M$ with $\epsilon > 0$ as small as we want.
Therefore we can extend any non-generate $2$-form defined near $\{0\} \times M$ to a
non-degenerate $2$-form on $W$ in a unique way up to homotopy relative to a
neighbourhood of $\{0\} \times M$.
\end{rem}

This is also true for $h$-cobordisms of dimension $\geq 6$ according to the
following lemma.

\begin{lem}\label{hlem}
Let $(W; M, M')$ be an $h$-cobordism of dimension $\geq 6$ with a
non-degenerate $2$-form $\eta$ defined near $M$. There is a non-degenerate
two-form $\omega$ on $W$ that coincides with $\eta$ near $M$. Moreover, 
the extension is unique up to homotopy relative to a neighbourhood of $M$.
\end{lem}

\begin{proof}
Let $(W'; M', M)$ be an inverse $h$-cobordism given by corollary \ref{invertibility},
so that $W \odot W' \simeq [0, 1] \times M$. By remark \ref{hlemprod}, there is a
non-degenerate $2$-form $\omega$ on $[0, 1] \times M$ which coincides with
$\eta$ near $\{0\} \times M$. Restricting $\omega$ to $W$ gives the required
extension. Now suppose that we have two non-degenerate $2$-forms $\omega$
and $\omega'$ on $W$ which coincide with $\eta$ near $M$. According to what
we have just proved, they both extend further to $W'$ because $W'$ is an
$h$-cobordism. Again by remark \ref{hlemprod}, $\omega$ and $\omega'$ are
homotopic on $[0, 1] \times M$ relative to a neighbourhood of $\{0\} \times M$,
in particular they are homotopic on $W$ relative to a neighbourhood of $M$.
\end{proof}

\section{Contact manifolds and Weinstein cobordisms}

Let $(M, \xi)$ be a \emph{contact manifold}, we mean $\xi$ is a cooriented hyperplane field
which is maximally non-integrable. We always endow $M$ with the orientation induced
by $\xi$. An \emph{exact symplectic manifold} is a manifold $V$ together with a $1$-form
$\lambda$ such that $\d \lambda$ is a symplectic form. There are at least two notions of
isomorphism between exact symplectic manifolds. If $(V, \lambda)$ and $(V', \lambda')$
are exact symplectic manifold, a diffeomorphism $\Psi : V \to V'$ is said to be:

\begin{itemize}
\item an \emph{exact symplectomorphism} if $\Psi^*\lambda' - \lambda$ is an exact $1$-form on $W$.
\item a \emph{symplectomorphism} if $\Psi^*\lambda' - \lambda$ is a closed $1$-form on $W$.
\end{itemize}

The \emph{symplectization} of a contact manifold $(M, \xi)$ is an exact symplectic manifold that can
be described as follows. The space of cotangent vectors of $M$ vanishing on $\xi$ is a one-dimensional
subbundle of the cotangent bundle $\T^*M$. Restricting our attention to non-zero cotangent vectors which
induce the right coorientation of $\xi$ yields a principal $\R_+^*$-bundle that we denote by $\S_\xi M$.
Since $\xi$ is cooriented, this bundle admits global sections which correspond to \emph{contact forms}
for $\xi$. In particular, $\S_\xi M$ is diffeomorphic to $\R \times M$. The canonical $1$-form $\lambda$
of $\T^*M$ induces a $1$-form denoted by $\lambda_\xi$ on $\S_\xi M$ called the \emph{Liouville form},
whose exterior derivative $\omega_\xi = \d \lambda_\xi$ is a symplectic form (this is equivalent to $\xi$
being a contact structure). The principal bundle structure can be recovered from the $1$-form
$\lambda_\xi$. Indeed, the \emph{Liouville vector field} $X_\xi$, defined by $X_\xi \lrcorner
\omega_\xi = \lambda_\xi$, is the infinitesimal generator of the $\R_+^*$-action. The flow
$\varphi_{X_\xi}^t$ of $X_\xi$ satisfies $(\varphi_{X_\xi}^t)^* \lambda_\xi = e^t \lambda_\xi$,
so it preserves $\ker \lambda_\xi$. Hence the projection map
\[(\S_\xi M / \R_+^*, \ker \lambda_\xi) \to (M, \xi)\] is a contactomorphism.
In particular, the symplectization $(\S_\xi M, \lambda_\xi)$ entirely recovers the contact
manifold $(M, \xi)$. In other words, any diffeomorphism
$\Psi : (\S_\xi M, \lambda_\xi) \to (\S_{\xi'} M', \lambda_{\xi'})$ such that $\Psi^*\lambda_{\xi'} = \lambda_\xi$
induces a contactomorphism $(M, \xi) \to (M', \xi')$. However, theorem \ref{main} shows that if $\S_\xi M$ and
$\S_{\xi'} M'$ are only exact symplectomorphic, then $M$ and $M'$ need not even be diffeomorphic.

\begin{rem}\label{splitsymp}
If we choose a contact form $\alpha$ for $\xi$, the symplectization naturally splits as:
\[(\S_\xi M, \lambda_\xi) = (\R \times M, e^t \alpha)\]
\end{rem}

A \emph{Weinstein structure} on a cobordism $(W; M, M')$ is a triple $(\omega, X, \phi)$ where
$(X, \phi)$ is a Morse pair and $\omega$ is a symplectic form (positive with respect to the orientation of $W$)
such that $X . \omega = \omega$. We call $X$ the \emph{Liouville vector field}. It gives rise to a \emph{Liouville form}
$\lambda=X \lrcorner \omega$. In fact, $(\omega, X)$ and $\lambda$ are equivalent pieces of data, often called
a \emph{Liouville structure}. The Liouville form $\lambda$ induces contact structures $\xi$ on $M$ and
$\xi'$ on $M'$ with contact forms $\alpha = \iota^*\lambda$ and $\alpha' = \iota'^*\lambda$, where $\iota: 
M \to W$ and $\iota': M' \to W$ are the inclusion maps. We sometimes say that $(W, \omega, X, \phi)$ is a
Weinstein cobordism from $(M, \xi)$ to $(M', \xi')$.

\begin{rem}\label{gluingweinstein}
Let $(W, \omega, X, \phi)$ be a Weinstein cobordism from $(M, \xi)$ to
$(M', \xi')$ and $(W', \omega', X', \phi')$ be a Weinstein cobordism from
$(M', \xi')$ to $(M'', \xi'')$. We now explain how to compose them in a Weinstein cobordism
from $(M, \xi)$ to $(M'', \xi'')$. Suppose that the Liouville forms $\lambda$
and $\lambda'$ induce the same contact form $\alpha'$ on $M'$. The flow of the Liouville vector
fields $X$ and $X'$ define collar neighbourhoods $[-\epsilon, 0] \times M'$ in $W$ and
$[0, \epsilon] \times M'$ in $W'$ where $\lambda$ and $\lambda'$ both read $e^{t'} \alpha'$ ($t'$ is the
coordinate in $\R$)  .
Using these collar neighbourhoods, we can glue $W$ and $W'$ along $M'$ and get a smooth
cobordism $(W\odot W'; M, M'')$ with a Liouville structure $(\omega'', X'')$ that restricts to $(\omega, X)$ and to
$(\omega', X')$ respectively on $W$ and $W'$. Even if $\phi = \phi'$ on $M'$, they do not
necessarily glue to a smooth function on $W \odot W'$. This can be arranged by composing
$\phi$ with a diffeomorphism of $W$ which is the identity on $M'$ and supported in an arbitrary small
neighbourhood of $M'$. For example, it is enough to arrange that $X . \phi = 1$ and $X' . \phi' = 1$ in a
neighbourhood of $M'$. Finally, we get a Weinstein cobordism $(W \odot W', \omega'', X'', \phi'')$ from $(M, \xi)$ to $(M'', \xi'')$.
\end{rem}

The easiest example is the following: let $M$ be a closed manifold together wih a contact form $\alpha$. For any two
smooth functions $f_-, f_+$ on $M$ with $\max f_- < \min f_+$, we consider the part of symplectization 
\[W=\{(t, x) \in \R \times M | f_-(x) \leq t \leq f_+(x)\}.\]
It admits a Liouville structure $(\omega=\d(e^t \alpha), X=\frac{\pd }{\pd t})$. By choosing a Morse function $\phi$
(constant on the boundary, as always) without critical points such that $X . \phi > 0$, we get a Weinstein cobordism
$(W, \omega, X, \phi)$.

\begin{rem}\label{contactforms}
If $(W; M, M')$ has a cobordism with Weinstein structure $(\omega, X, \phi)$. This induces contact forms
$\alpha$ and $\alpha'$ respectively on $M$ and $M'$. By multiplying $\omega$ by a positive number, and
composing with parts of symplectizations as above, we can change the contact forms $\alpha$ and $\alpha'$
for any contact forms $e^k e^{-f} \alpha$ and $e^k e^{f'} \alpha'$ with $k \in \R$, and smooth functions $f : M \to [0, + \infty[$ and
$f' : M' \to [0, +\infty[$.
\end{rem}

A \emph{Weinstein homotopy} on $W$ is a smooth path $(\omega_s, X_s, \phi_s)$, such that
$(X_s, \phi_s)$ is a Morse homotopy and for all but finitely many parameters $s$ (where $(X_s, \phi_s)$
encounters a birth-death singularity) $(\omega_s, X_s, \phi_s)$ is a Weinstein structure.

For a cobordism to admit a Weinstein structure, it is necessary that it carries a non-degenerate $2$-form.
But there are more severe topological constraints due to the following (see \cite{CE2012} p.242 for a proof).

\begin{prop}
If $(W, \omega, X, \phi)$ is a Weinstein cobordism of dimension $2n$, then the critical points
of $\phi$ have index $\leq n$.
\end{prop}

A Weinstein cobordism $(W, \omega, X, \phi)$ of dimension $2n$ is called \emph{subcritical} if 
the critical points of $\phi$ have index $< n$. It is known for some time that subcritical Weinstein
cobordisms exhibit remarkable flexibility properties (see \cite{MR1461569}). Yet a larger class of
Weinstein cobordisms with flexibility properties was recently discovered. A Weinstein cobordism
$(W, \omega, X, \phi)$ is called \emph{flexible} if it is the composition of finitely many Weinstein
cobordisms $(W^i, \omega^i, X^i, \phi^i)$ which are elementary (that is $X^i$ has no trajectory 
joining critical points) and whose attaching spheres of lagrangian handles form a \emph{loose}
legendrian link in the lower boundary of $W^i$ (see \cite{CE2012} p.250-251). Notice that it is
clear from the definition that the composition of two flexible Weinstein cobordisms is
still a flexible Weinstein cobordism.

We now state two theorems about flexible Weinstein structures that are relevant to
our purpose (\cite{CE2012}, p.279).

\begin{thm} [Cieliebak, Eliashberg]\label{weinsteinexist}
Let $(W; M, M')$ be a cobordism of dimension $2n\geq 6$ together with
a non-degenerate $2$-form $\eta$ and a Morse pair $(Y, \phi)$ with critical points of
index $\leq n$ such that $(\eta, Y, \phi)$ is a Weinstein structure near $M$.
Then there is a flexible Weinstein structure $(\omega, X, \phi)$ on $W$ such that 
$\omega = \eta$ near $M$.
\end{thm}

\begin{thm}[Cieliebak, Eliashberg]\label{weinsteinuniq}
Let $(W; M, M')$ be a cobordism of dimension $2n \geq 6$ together with a flexible
Weinstein structure $(\omega, X, \phi)$. Then for any Morse homotopy $(Y_s, \phi_s)$ fixed near
the boundary, with critical points of index $\leq n$, such that $(Y_0, \phi_0) = (X, \phi)$,
there is a Weinstein homotopy $(\omega_s, X_s, \phi_s)$ satisfying:
\begin{itemize}
\item $(\omega_0, X_0, \phi_0) = (\omega, X, \phi)$
\item $(X_s, \phi_s)$ is fixed near $\pd W$, $\omega_s$ is fixed near $\pd_-W$ and
$\omega_s = e^{c_s} \omega_0$ near $\pd_+W$ for a smooth real-valued function $s \mapsto c_s$.
\end{itemize}
\end{thm}

\section{Main results}
\subsection{Symplectomorphic symplectizations}

We start by a lemma which shows that theorem \ref{weinsteinexist}
can be applied to any $h$-cobordism of dimension $\geq 6$ from a closed contact manifold.

\begin{lem}\label{hcobweinstein}
Let $(M,\xi)$ be a closed contact manifold of dimension $\geq 5$ and let $(W; M, M')$ be
an $h$-cobordism. Then there is a flexible Weinstein structure $(\omega, X, \phi)$ on $W$
that induces a contact structure isotopic to $\xi$ on $M$ and which has only critical points
of index $2$ and $3$.
\end{lem}

\begin{proof}
Take a collar neighbourhood $[0, \epsilon] \times M$ of $M$ in $W$. Consider the standard
Weinstein structure $(\d(e^t \alpha), \frac{\pd}{\pd t}, t)$ in this collar. By lemma \ref{hlem},
the $2$-form $\d(e^t\alpha)$ extends to $W$ as a non-degenerate $2$-form. By lemma
\ref{normalform}, the Morse pair $(\frac{\pd}{\pd t}, t)$ extends to a Morse pair $(Y, \phi)$ on $W$
with only critical points of index $2$ and $3$. We now apply theorem \ref{weinsteinexist} to get
a flexible Weinstein structure $(\omega, X, \phi)$ such that $\omega = \eta$ near $M$.
Then the induced contact structure on $M$ is isotopic to $\xi$.
\end{proof}

\begin{rem}\label{identification}
By Gray's stability theorem, any two isotopic contact structures are contactomorphic. So
after applying lemma \ref{hcobweinstein}, we may compose the identification of $\pd_- W$
with $M$ by such a contactomorphism to actually get a Weinstein cobordism from $(M, \xi)$.
We will do this implicitly in the proof of theorem \ref{mazurtricksymp} below.
\end{rem}

We now turn to our main result which can be thought of as a symplectic analogue of
corollary~\ref{mazurtrick}.

\begin{thm}\label{mazurtricksymp}
Let $(M, \xi)$ be a closed contact manifold of dimension $\geq 5$. Then for any $h$-cobordism
$(W; M, M')$ there is a contact structure $\xi'$ on $M'$ such that $(\S_\xi M, \lambda_\xi)$ and
$(\S_{\xi'} M', \lambda_{\xi'})$ are exact symplectomorphic.
\end{thm}

\begin{proof}
Let $(W'; M', M)$ be an inverse $h$-cobordism of $(W; M, M')$ given by corollary \ref{invertibility}.
By lemma \ref{hcobweinstein}, there is a flexible Weinstein structure $(\omega, X, \phi)$ on $W$ which
induces the contact structure $\xi$ on $M$. It also induces a contact structure $\xi'$ on $M'$.
Again by lemma \ref{hcobweinstein}, there is a flexible Weinstein structure $(\omega', X', \phi')$ on $W'$
that induces the contact structure $\xi'$ on $M'$. Denote by $\alpha$ and $\alpha'$ the contact
forms respectively on $M$ and $M'$ induced by $(W, \omega, X, \phi)$. According to remark
\ref{contactforms}, we can arrange $W'$ so that the contact form induced on $M'$ equals $\alpha'$.
Up to composing $\phi$ and $\phi'$ by affine transformations of $\R$, we can assume that $\phi = 0$
on $M$, $\phi=1$ on $M'$, $\phi' = 1$ on $M'$ and $\phi' = 2$ on $M$. After arranging the functions
$\phi$ and $\phi'$ as in remark \ref{gluingweinstein}, we can compose $W$ and $W'$ to get a smooth
cobordism $W'' = W \odot W'$ together with a Weinstein structure $(\omega'', X'', \phi'')$ which restricts to
$(\omega, X, \phi)$ on $W$ and to $(\omega', X', \phi')$ on $W'$. The function $\phi''$ has only critical
points of index $2$ and $3$. Since $W \odot W'$ is diffeomorphic to a product cobordism, lemma
\ref{cancelling} implies that there is a Morse homotopy $(Y_s, \phi''_s)$ fixed near the boundary such
that $(Y_0, \phi''_0) = (X'', \phi'')$ and $\phi''_1$ has no critical points. Now by theorem
\ref{weinsteinuniq}, there is a Weinstein homotopy $(\omega''_s, X''_s, \phi''_s)$ such that:
\begin{itemize}
\item $(\omega''_0, X''_0, \phi''_0) = (\omega'', X'', \phi'')$
\item $(X''_s, \phi_s)$ is fixed near $\pd W''$, $\omega''_s$ is fixed near $\pd_-W''$ and
$\omega_s = e^{c_s} \omega''_0$ near $\pd_+W''$ for a smooth real-valued function $s \mapsto c_s$.
\end{itemize}

Near $\pd_+ W''$, $X''_s$ is fixed and $\omega''_s$ is fixed up to a constant, so in particular, the contact
structure $\xi''_s$ induced on $\pd_+ W'' = M$ is fixed during the homotopy. The holonomy of the
Liouville vector field $X''_1$ defines a contactomorphism $(M, \xi)$ to $(M, \xi''_1)$. In the cobordism $W'$,
we now change the identification of $\pd_+ W'$ with $M$ by composing it with this contactomorphism (as in
remark \ref{identification}), so that the contact structure on $M$ induced by $(W', \omega', X', \phi')$
is equal to $\xi$. According to remark \ref{contactforms}, we may compose $W'$ with a part of the
symplectization of $M$ so that it induces the contact forms $e^k \alpha$ for some $k > 0$. The Weinstein
homotopy $(\omega''_s, X''_s, \phi''_s)$ obviously extends to this slightly enlarged cobordism since
$(X''_s, \phi''_s)$ is fixed near $\pd_+ W''$ and $\omega''s= e^{c_s} \omega''_0$ near $\pd_+ W''$.
Up to composing $\phi''_s$ with a diffeomorphism of $\R$, assume that $\phi''_s = 2$ on $\pd_+ W''$
still holds.

In the spirit of the proof of corollary \ref{mazurtrick}, we will construct an exact symplectic manifold
$V$ by gluing infinitely many copies of $W$ and $W'$ and show that $V$ is exact symplectomorphic
to both $\S_\xi M$ and $\S_{\xi'} M'$.

We now define translates of $W$ and $W'$ as follows, for $j \in \Z$:

\[(W^j, \omega^j, X^j, \phi^j) = (W, e^{jk} \omega, X, \phi + 2j) \text{ and }
(W'^j, \omega'^j, X'^j, \phi^j) = (W', e^{jk} \omega', X', \phi' + 2j),\]

and consider:
\[V =\cdots \odot W^{-1} \odot W'^{-1} \odot W^0 \odot W'^0 \odot W^1 \odot W'^1 \odot \cdots\]

According to remark \ref{gluingweinstein}, this is well-defined and carries a Weinstein structure
$(\omega, X, \phi)$ that restricts to the given one on each $W_i$ and $W'_i$.

We now prove that $V$ is exact symplectomorphic to $\S_\xi M$.

We want to repeat the homotopy $(\omega''_s, X''_s, \phi''_s)$ on the whole $V$ by translation.
We just need to take care of the scaling factor $e^{c_s}$ near the top boundary. So define, for $j \in \Z$,
on $W^j \odot W'^j$:
\[(\omega_s, X_s, \phi_s) = (e^{j c_s} e^{jk} \omega''_s, X''_s, \phi''_s +2j).\]

This gives a Weinstein homotopy of $V$ during which the vector field $X_s$ is complete (it is invariant by
translation in $j$) and is transverse to the hypersurfaces $M^{j}=\phi_s^{-1}(2j)=\phi^{-1}(2j) \simeq M$
for all $j \in \Z$. Note that this homotopy is fixed near $\phi^{-1}(0) \simeq M$ (we will make use of this
in section 4.2).

We now look for an isotopy $\Psi_s$ of $V$ such that $\Psi_s^* \lambda_s - \lambda_0$ is exact
(here $\lambda_s = X_s \lrcorner \omega_s$). We will find it using Moser's lemma (see \cite{CE2012}
p. 240-241 for a similar argument). Take $C > \max(0, \max c_s)$ and consider
$\tilde{M^j} = \varphi^{jC}_{X_0} (M^j)$
($\varphi_X^t$ denotes the flow at time $t$ of a vector field $X$).

Since $X_s$ is complete for all $s \in [0,1]$, we can define:

\[\Theta^{2j}_s = \varphi^{j(C-c_s)}_{X_s} \circ \varphi_{X_0}^{-jC} : \tilde{M^j} \to V\]
And we have:
\begin{align*}
(\Theta_s^{2j})^*\lambda_s &= (\varphi_{X_0}^{-jC})^* \circ (\varphi^{j(C-c_s)}_{X_s})^*  (\lambda_s)\\
& = (\varphi_{X_0}^{-jC})^* (e^{- j(C-c_s)} \lambda_s) \\
& = (\varphi_{X_0}^{-jC})^* (e^{jC} \lambda_0) \\
& = \lambda_0
\end{align*}

We can extend $\Theta_s^{2j}$ near $\tilde{M^j}$ in a unique way so that $(\Theta_s^{2j})^* \lambda_s = \lambda_0$.
The image of $\Theta_s^{2j}$ is $\varphi_{X_s}^{j(C-c_s)}(M^j)$, so they are all disjoint. Hence we can find an isotopy
$\Theta_s : V \to V$ that coincides with $\Theta_s^{2j}$ near $\tilde{M^j}$ for all $j$. The path $\Theta_s^* \lambda_s$ is
now fixed near each $\tilde{M^j}$ and Moser's lemma applied to each region between $\tilde{M^j}$
and $\tilde{M^{j+1}}$ gives an isotopy $\Psi^s : V \to V$ such that $\Psi_s ^*\lambda_s - \lambda_0$
is exact.

Since $X_1$ is complete and nowhere vanishing, its flow defines a diffeomorphism $\Xi : \R \times M \to V$
which satisfies $\Xi^*\lambda_1 = e^{t} \alpha$. The map $\Xi^{-1} \circ \Theta_1$ is the required exact
symplectomorphism from $(V, \lambda)$ to $(\S_\xi M, \lambda_\xi)  = (\R \times M, e^t \alpha)$.

Since $(W'^{-1} \odot W^0; M', M')$ is a product cobordism, we can apply exactly the same reasoning and
find another Weinstein homotopy of $V$, which we then turn into an exact symplectomorphism from $(V, \lambda)$ to
$(\S_{\xi'} M', \lambda_{\xi'})$.
\end{proof}

Lemma \ref{hcobweinstein} and theorem \ref{mazurtricksymp} all together imply the theorem \ref{main}
stated in the introduction.

\begin{rem}
Given a closed contact manifold $(M, \xi)$ of dimension $\geq 5$, we have associated to any $h$-cobordism
from $M$ a contact manifold $(M', \xi')$ such that $\S_\xi M$ and $\S_{\xi'} M'$ are exact symplectomorphic.
So by the $s$-cobordism theorem, this produces as many contact manifolds as the cardinality of $\Wh(\pi_1 M)$.
Of course, this is only interesting when $\Wh(\pi_1 M) \neq 0$. Note that the example given in the introduction
together with $s$-cobordism theorem shows that $\Wh(\Z/7\Z) \neq 0$ (see \cite{MR0362320} p.42-45 for more
examples of non-trivial Whitehead groups).
\end{rem}

\subsection{Contact manifolds at infinity of Weinstein and Stein manifolds}

A Weinstein structure on an open manifold $V$ is a triple $(\omega, X, \phi)$ where $\omega$ is a symplectic
form, $X$ is a complete vector field such that $X . \omega = \omega$, $\phi$ is a Morse function on $V$
(proper and bounded from below) for which $X$ is a pseudo-gradient vector field. Notice that the region between two
regular values of $\phi$ is a Weinstein cobordism in the sense of section $3$. We call $(V, \omega, X, \phi)$ of \emph{finite
type} if there is $c>0$ such that $\phi^{-1}([c, +\infty[)$ does not contain any critical point. In this case, the level sets
of $\phi$ above $c$ are all contactomorphic by flowing along the Liouville vector field $X$, we call it the \emph{contact
manifold at infinity} of $(V, \omega, X, \phi)$. This depends only on $(\omega, X)$ and we may think that it
is actually independent of $X$ (see \cite{CE2012} p.238-239). As a corollary of the proof of theorem \ref{mazurtricksymp},
we show that this is not the case.

We need the following notion of homotopy for open weinstein manifold (see \cite{CE2012} p.246).
A \emph{weinstein homotopy} on $V$ is a smooth path $(\omega_s, X_s, \phi_s)_{s \in [0,1]}$ of Weinstein
structures such that $(X_s, \phi_s)$ is a generic path (it encounters only birth-death type singularities),
there is a subdivision $0 = a_0 < a_1 < \cdots < a_p = 1$, and for each $i \in \{0, \cdots, p-1\}$ an increasing
sequence $(c^i_k) \tend_{k \to + \infty} + \infty$ of regular values of $\phi_s$ for all $s \in [a_i, a_{i + 1}]$.
This definition prevents critical points to escape at infinity during a Weinstein homotopy.

\begin{cor}\label{mazurweinstein}
Let $(V, \omega, X, \phi)$ be a finite type Weinstein manifold of dimension $\geq 6$ with contact
manifold at infinity contactomorphic to $(M, \xi)$. For any $h$-cobordism $(W; M, M')$ there is a
Weinstein homotopy $(\omega_s, X_s, \phi_s)_{s \in [0,1]}$ such that $(\omega_0, X_0, \phi_0)
= (\omega, X, \phi)$ and $(W, \omega_1, X_1, \phi_1)$ is a finite type Weinstein manifold with
contact manifold at infinity diffeomorphic to $M'$.
\end{cor}

\begin{proof}
Let $c$ be sufficiently close to $+ \infty$ so that $\phi$ has no critical points in $\{\phi \geq c\}$. Then $\phi^{-1}(c)$ is
contactomorphic to $(M, \xi)$ and the flow of $X$ identifies $\{\phi \geq c\}$ with $[0, +\infty[ \times M$.
The proof of theorem \ref{mazurtricksymp} shows that there is a Weinstein homotopy $(\omega_s, X_s, \phi_s)$ on
$[0, + \infty[ \times M$ such that:
\begin{itemize}
\item $(\omega_0, X_0, \phi_0) = (\omega, X, \phi)$
\item $(\omega_s, X_s, \phi_s)$ is fixed near $\{0\} \times M$
\item For $c' > 0$ sufficiently large, $\{\phi_1 \geq c'\}$ contains no critical points of $\phi_1$
and $\phi_1^{-1}(c')$ is diffeomorphic to $M'$.
\end{itemize}
We extend the Weinstein homotopy by a constant homotopy on $\{\phi \leq c\} = \{\phi_s \leq c\}$ to get the
result.
\end{proof}

\begin{rem}
\begin{enumerate}
\item If $M$ and $M'$ are not diffeomorphic, critical points have to appear out of every compact set
during the Weinstein homotopy in corollary \ref{mazurweinstein} because otherwise the topology of the
contact manifold at infinity would not change.
\item The Weinstein homotopy can be made fixed on an arbitrary large compact set of $V$: in some sense,
it only move things at infinity.
\item According to the proof of theorem \ref{mazurtricksymp}, we can find an isotopy $\Psi_s$
of $V$ such that $\Psi_s^* \lambda_s = \lambda_0 + \d f_s$. In particular, we get a Weinstein
homotopy $(\omega_0, \Psi_s^* X_s, \Psi_s^*\phi_s)$ with fixed symplectic form during which the topology
of the contact manifold at infinity changes.
\item Since the homotopy in corollary \ref{mazurweinstein} only concerns the cylindrical end $[0, +\infty[ \times M$.
This applies of course to any symplectic manifold with cylindrical end, not necessarily Weinstein.
\end{enumerate}
\end{rem}

And finally using the Weinstein-Stein correspondence from \cite{CE2012}, we can give a corollary
concerning the complex geometry of Stein manifolds.

\begin{cor}\label{mazurstein}
Let $(V, J, \phi)$ be a finite type Stein manifold of dimension $\geq 6$ with contact manifold at infinity
contactomorphic to $(M, \xi)$. For any $h$-cobordism $(W; M, M')$, there is a Stein homotopy
$(J, \phi_s)_{s \in [0,1]}$ such that $\phi_0 = \phi$ and $(V, J, \phi_1)$ is a finite type Stein manifold
with contact manifold at infinity diffeomorphic to $M'$.
\end{cor}

\begin{proof}
In the spirit of \cite{CE2012}, the proof goes from Stein to Weinstein and back. Let $(\omega = -\d \d^c \phi,
X= \nabla_{\phi} \phi, \phi)$ be the Weinstein structure associated to $(V, J, \phi)$
(see \cite{CE2012}, p.244-245). By corollary \ref{mazurweinstein}, there is a Weinstein homotopy
$(\omega_s, X_s, \phi_s)$ such that $(\omega_0, X_0, \phi_0) = (\omega, X, \phi)$ and level sets of
$\phi_1$ at infinity are diffeomorphic to $M'$. Now by theorem 15.3 in \cite{CE2012}, there is an isotopy
$\Psi_s$ of $V$ and an isotopy $g_s$ of $\R$ such that $(J, g_s \circ \phi_s \circ \Psi_s^{-1})$ is a Stein homotopy.
The level sets at infinity of $g_1 \circ \phi_1 \circ \Psi_1^{-1}$ are then diffeomorphic to $M'$.
\end{proof}

\section{Questions}
We now state a few questions that remain open.
\begin{enumerate}
\item Does there exist contact structures $\xi$ and $\xi'$ on a closed manifold $M$ that are not contactomorphic but
whose symplectizations $\S_\xi M$ and $\S_{\xi'} M$ are (exact) symplectomorphic? There are many examples of closed
manifolds $M$ of dimension $\geq 5$ for which there are non-trivial $h$-cobordisms from $M$ to itself
(see \cite{MR0196736}). A flexible Weinstein structure on such a cobordism gives two contact structures on $M$
whose symplectizations are exact symplectomorphic according to theorem \ref{mazurtricksymp} but we do not
know if they are contactomorphic or not.
\item What about contact three-manifolds? At present, no examples of non-trivial smooth $4$-dimensional
$h$-cobordisms are known (see the discussion in \cite{MR2228319}). So the method developped in this
paper will hardly apply.
\end{enumerate}

\bibliographystyle{amsalpha}
\bibliography{/Users/sylvaincourte/Dropbox/Mathematiques/Latex/biblio}

\end{document}